\documentclass[11pt,leqno]{article}

\usepackage{amsthm,amsfonts,amssymb,amsmath,oldgerm, soul}
\usepackage{fullpage}
\usepackage{graphicx}
\usepackage{mathrsfs}
\usepackage[dvipsnames]{xcolor}

\numberwithin{equation}{section}

\renewcommand{\d}{\partial}
\newcommand{\eps}{\varepsilon}
\newcommand{\N}{\mathbb{N}}

\renewcommand{\Re}{\text{Re}}

\newcommand\LA{\left\langle}
\newcommand\RA{\right\rangle}

\newcommand{\RM}{{\mathbb{R}}}
\newcommand{\CM}{{\mathbb{C}}}
\newcommand{\NM}{{\mathbb{N}}}
\newcommand{\ZM}{{\mathbb{Z}}}
\newtheorem{theorem}{Theorem}[section]
\newtheorem{proposition}[theorem]{Proposition}
\newtheorem{corollary}[theorem]{Corollary}
\newtheorem{lemma}[theorem]{Lemma}
\newtheorem{remark}[theorem]{Remark}
\theoremstyle{definition}
\newtheorem{definition}[theorem]{Definition}

\allowdisplaybreaks[3]

\title{Subharmonic Dynamics of Wave Trains in Reaction Diffusion Systems}

\author{Mathew~A.~Johnson\thanks{Department of Mathematics, University of Kansas, 1460 Jayhawk Boulevard, 
Lawrence, KS 66045, USA; matjohn@ku.edu}\quad\&\quad Wesley R. Perkins\thanks{Department of Mathematics, University of Kansas, 1460 Jayhawk Boulevard, 
Lawrence, KS 66045, USA; wesley.perkins@ku.edu} }

\date{\today}

\begin{document}

\maketitle

\begin{abstract}
We investigate the stability and nonlinear local dynamics of spectrally stable wave trains in reaction-diffusion systems.   
For each $N\in\NM$, such $T$-periodic traveling waves are easily seen to be nonlinearly asymptotically stable (with asymptotic phase)
with exponential rates of decay when subject to $NT$-periodic, i.e., subharmonic, perturbations.  However, both the allowable size of perturbations and the exponential
rates of decay depend on $N$, and, in particular, they tend to zero as $N\to\infty$, leading to a lack of uniformity in such subharmonic stability results.
In this work, we build on recent work by the authors and introduce a methodology that allows us to achieve a stability result for subharmonic perturbations which is uniform in $N$.
Our work is motivated by the dynamics of such waves when subject to perturbations which are localized (i.e. integrable on the line), which has recently received considerable
attention by many authors.
\end{abstract}

\section{Introduction}\label{S:intro}

In this work, we consider the local dynamics of periodic traveling wave solutions, i.e. wave trains, in reaction diffusion systems of the form
\begin{equation}\label{e:rd}
    u_t = u_{xx} + f(u),~~x\in\RM,~~t\geq 0,~~u\in\RM^n
\end{equation}
where $n\in\NM$ and $f:\RM^n\to\RM^n$ is a $C^K$-smooth nonlinearity for some $K\geq 3$.
Such systems arise naturally in many areas of applied mathematics, and the behavior of such wave train solutions when subject to a variety of classes of perturbations has been studied intensively
over the last decade.  Most commonly in the literature, one studies the stability and instability of such periodic traveling waves to perturbations
which are \emph{localized}, i.e. integrable on the line, or which are \emph{nonlocalized}, accounting for asymptotic phase differences at infinity.  
See, for example, \cite{DSSS,JNRZ_13_1,JNRZ_13_2,JZ_11_1,SSSU,SW15} and references therein.

Here, we consider the stability and long-time dynamics of $T$-periodic traveling wave solutions of \eqref{e:rd} when subjected to 
$NT$-periodic, i.e. \emph{subharmonic}, perturbations for some $N\in\NM$.  More precisely, suppose that $u(x,t)=\phi(k(x-ct))$ is a periodic
 traveling wave solution of \eqref{e:rd} with period $T=1/k$, where we choose
$k\in\RM$ so that the profile $\phi\in H^1_{\rm loc}(\RM)$ is a $1$-periodic stationary solution of
\begin{equation}\label{e:RDE_trav}
    k u_t - kc u_x = k^2 u_{xx} + f(u),
\end{equation}
i.e. it satisfies the profile equation
\begin{equation}\label{e:profile}
k^2\phi''+kc\phi'+f(\phi)=0.
\end{equation}
Given such a solution, note that a function of the form $u(x,t)=\phi(x)+v(x,t)$ is a solution of \eqref{e:RDE_trav}
provided it satisfies a system of the form
\begin{equation}\label{e:rd_nonlinear}
kv_t=k\mathcal{L}[\phi]v+\mathcal{N}(v),
\end{equation}
where here $\mathcal{N}(v)$ is at least quadratic in $v$ and $\mathcal{L}[\phi]$ is the linear differential operator
\[
k\mathcal{L}[\phi]:=k^2\partial_x^2+kc\partial_x+Df(\phi).
\]
Naturally, the domain of the operator $\mathcal{L}[\phi]$ is determined by the chosen class of perturbations $v$ of the underlying standing wave $\phi$
and, as mentioned above, several choices are available in the literature.
As we are interested in subharmonic perturbations, i.e. perturbations with period $N\in\NM$, we consider $\mathcal{L}[\phi]$ as a closed, densely defined linear
operator acting on $L^2_{\rm per}(0,N)$ with $1$-periodic coefficients.

The stability analysis of periodic waves to such subharmonic perturbations naturally relies on a detailed understanding of the 
spectrum of $\mathcal{L}[\phi]$ acting on $L^2_{\rm per}(0,N)$.  
To describe the $N$-periodic spectrum of $\mathcal{L}[\phi]$, we begin by introducing the notion of spectral stability
that will be used throughout this work.

\begin{definition}\label{D:specstab}
A $1$-periodic stationary solution $\phi\in H^1_{\rm loc}(\RM)$ of \eqref{e:RDE_trav} is said to be \emph{diffusively spectrally stable} provided the following conditions hold:
\begin{itemize}
\item[(i)] The spectrum of the linear operator $\mathcal{L}[\phi]$ acting on $L^2(\RM)$ satisfies
\[
\sigma_{L^2(\RM)}\left(\mathcal{L}[\phi]\right)\subset\left\{\lambda\in\CM:\Re(\lambda)<0\right\}\cup\{0\};
\]
\item[(ii)] There exists a $\theta>0$ such that for any $\xi\in[-\pi,\pi)$ the real part of the spectrum of the Bloch operator $\mathcal{L}_\xi[\phi]$ acting
on $L^2_{\rm per}(0,1)$ satisfies
\[
\Re\left(\sigma_{L^2_{\rm per}(0,1)}\left(\mathcal{L}_\xi[\phi]\right)\right)\leq-\theta\xi^2;
\]
\item[(iii)] $\lambda=0$ is a simple eigenvalue of $\mathcal{L}_0[\phi]$ with associated eigenfunction $\phi'$.
\end{itemize}
\end{definition}

Since the pioneering work of Schnieder \cite{S96,S98_1}, the above notion of spectral stability has been taken as the standard spectral assumption in nonlinear stability
results for periodic traveling/standing waves in reaction diffusion systems.   Specifically,
the above notion of spectral stability is sufficiently strong to allow one to immediately conclude important details regarding the nonlinear dynamics of $\phi$
under localized, or general bounded, perturbations, including long-time asymptotics of the associated modulation functions.  
For more information, see \cite{DSSS,JNRZ_13_1,JNRZ_13_2,SSSU,SW15} and references therein.

\begin{remark}
Note the assumption on simplicity of the eigenvalue $\lambda=0$ is natural since such periodic standing waves typically appear as one-parameter families
parametrized only by translational invariance.  Indeed, solutions of \eqref{e:profile} are readily seen to rely (up to translation invariance) on the $n+2$ parameters $(\phi(0),k,c)$,
while periodicity requires the enforcement of $n$ constraints, leaving in general a two-parameter family of $1$-periodic solutions
\[
u(x-ct-x_0;c,x_0)
\]
which satisfy \eqref{e:profile} with $k=k(c)$.  Due to the secular dependence of the frequency $k$ on the wave speed $c$, variations in $c$ do not preserve
periodicity and hence, generically, it follows one should expect the kernel of $\mathcal{L}[\phi]$ to be one-dimensional, which leads to (iii) in
Definition \eqref{D:specstab} above.
\end{remark}

Given a diffusively spectrally stable $1$-periodic traveling wave solution $\phi$ of \eqref{e:rd}, one can now easily characterize the spectrum of $\mathcal{L}[\phi]$
acting on $L^2_{\rm per}(0,N)$.  Indeed, as described in Section \ref{S:bloch} below, the spectrum of $\mathcal{L}[\phi]$ acting on $L^2_{\rm per}(0,N)$
is equal to the union of the necessarily discrete\footnote{Note since the domains of the operators $\mathcal{L}_\xi[\phi]$ are compactly contained
in $L^2_{\rm per}(0,1)$, it follows that their $L^2_{\rm per}(0,1)$-spectrum is comprised entirely of isolated eigenvalues
with finite multiplicities.} spectrum of the corresponding Bloch operators $\mathcal{L}_\xi[\phi]$, defined in Definition \ref{D:specstab} above,
acting in $L^2_{\rm per}(0,1)$ for the discrete (finite) subset of $\xi\in[-\pi,\pi)$ such that $e^{i\xi N}=1$.  It follows that diffusively spectrally stable
periodic traveling waves of \eqref{e:rd} are necessarily spectrally stable to all subharmonic perturbations.  In particular, for each $N\in\NM$ the
non-zero $N$-periodic eigenvalues of $\mathcal{L}[\phi]$ satisfy the spectral gap condition
\[
\Re\left(\sigma_{L^2_{\rm per}(0,N)}\left(\mathcal{L}[\phi]\right)\setminus\{0\}\right)\leq -\delta_N
\]
for some constant $\delta_N>0$.  From here, using that $\mathcal{L}[\phi]$ is sectorial, it is easy to show that for each $\delta\in(0,\delta_N)$ there exists a constant
$C_\delta>0$ such that
\begin{equation}\label{e:lin_exp}
\left\|e^{\mathcal{L}[\phi]t}\left(1-\mathcal{P}_1\right)f\right\|_{L^2_{\rm per}(0,N)}\leq C_\delta e^{-\delta t}\|f\|_{L^2_{\rm per}(0,N)}.
\end{equation}
for all $f\in L^2_{\rm per}(0,N)$, where here $\mathcal{P}_1$ denotes the projection of $L^2_{\rm per}(0,N)$ onto the $N$-periodic kernel of $\mathcal{L}[\phi]$
spanned by $\phi'$.  Equipped with this linear estimate, one can now establish the following nonlinear stability result.

\begin{proposition}\label{P:sub_stab}
Let $\phi\in H^1_{\rm loc}$ be a $1$-periodic stationary solution of \eqref{e:RDE_trav} and fix $N\in\NM$.  Assume that $\phi$ is diffusively spectrally stable, in the
sense of Definition \ref{D:specstab} below and, for each $N\in\NM$, take $\delta_N>0$ such that
\begin{equation}\label{spec_gap}
\max \Re\left(\sigma_{L^2_{\rm per}(0,N)}\left(\mathcal{L}[\phi]\right)\setminus\{0\}\right)=-\delta_N
\end{equation}
holds.  Then for each $N\in\NM$, $\phi$ is asymptotically stable
to subharmonic $N$-periodic perturbations.  More precisely, for every $\delta\in(0,\delta_N)$ there exists an $\eps=\eps_\delta>0$ and a constant
$C=C_\delta>0$ such that whenever $u_0\in H^1_{\rm per}(0,N)$ and $\|u_0-\phi\|_{H^1(0,N)}<\eps$, then the solution $u$ of \eqref{e:RDE_trav} with
initial data $u(0)=u_0$ exists globally in time and satisfies
\[
\left\|u(\cdot,t)-\phi(\cdot+\sigma_\infty)\right\|_{H^1(0,N)}\leq Ce^{-\delta t}\|u_0-\phi\|_{H^1(0,N)}
\]
for all $t>0$, where here $\sigma_\infty=\sigma_\infty(N)$ is some constant.
\end{proposition}

The proof of Proposition \ref{P:sub_stab} is by now standard, and can be completed by following appropriate texts: see, for example, \cite[Chapter 4]{KP_book}.
The main idea is that the linear estimate \eqref{e:lin_exp} suggests that if $u(x,t)$ is a solution of \eqref{e:RDE_trav} which is initially close 
to $\phi$ in $L^2_{\rm per}(0,N)$, then there exists a (small) time-dependent modulation function $\sigma(t)$ such that $u(x,t)$ essentially behaves
for large time as
\[
u(x,t)\approx \phi(x)+\sigma(t)\phi'(x)\approx \phi(x+\sigma(t)),
\]
corresponding to standard asymptotic (orbital) stability of $\phi$.  With this insight gained from  \eqref{e:lin_exp}, 
a straightforward nonlinear iteration scheme completes the proof of Proposition \ref{P:sub_stab}.

While Proposition \ref{P:sub_stab} establishes nonlinear stability of $\phi$ in $L^2_{\rm per}(0,N)$ for each fixed $N\in\NM$, it lacks uniformity in $N$ in two important (and related) aspects.
Indeed, note that the exponential rate of decay $\delta$ and the allowable size of initial perturbations $\eps=\eps_\delta$ are both controlled completely in terms
of the size of the spectral gap $\delta_N>0$.  Since $\delta_N\to 0$ as $N\to\infty$, it follows that both $\delta$ and $\eps$ chosen in Proposition \ref{P:sub_stab}
necessarily tend to zero\footnote{Additionally, this degeneracy can be seen in the linear estimate \eqref{e:lin_exp} since both $\delta\to 0^+$ and $C_\delta\to\infty$ as $N\to\infty$.} 
as $N\to\infty$.  With this observation in mind, it is natural to ask if one can obtain a stability result to $N$-periodic perturbations which
is uniform in $N$.  In such a result, one should naturally require that both the rate of decay and and the size of initial perturbations be independent of $N$, thus 
depending only on the background wave $\phi$.  This is precisely achieved in our main result.

\begin{theorem}[Uniform Subharmonic Asymptotic Stability]\label{T:main}  Fix\footnote{Here and throughout, $K$ encodes the regularity
of the nonlinearity $f$ in \eqref{e:rd}.} $K\geq 3$.
Suppose $\phi\in H^1_{\rm loc}(\RM)$ is a $1$-periodic stationary solution of \eqref{e:RDE_trav} that is diffusively spectrally stable, in the sense of Definition \ref{D:specstab}.  
There exists an $\eps>0$ and a constant $C>0$ such that, for each $N\in\NM$,
 whenever $u_0\in L^1_{\rm per}(0,N)\cap H^K_{\rm per}(0,N)$ and 
\[
E_0:=\left\|u_0-\phi\right\|_{L^1_{\rm per}(0,N)\cap H^K_{\rm per}(0,N)}<\eps, 
\]
 there exists a function $\widetilde{\psi}(x,t)$ satisfying $\widetilde{\psi}(\cdot,0)\equiv 0$ such that the solution of \eqref{e:RDE_trav} with initial
 data $u(0)=u_0$ exists globally in time and satisfies
\begin{equation}\label{e:result1}
\left\|u\left(\cdot-\widetilde{\psi}(\cdot,t),t\right)-\phi\right\|_{H^K_{\rm per}(0,N)},~~\left\|\nabla_{x,t}\widetilde{\psi}(\cdot,t)\right\|_{H^K_{\rm per}(0,N)}\leq C E_0(1+t)^{-3/4}
\end{equation}
for all $t\geq 0$.  Further, there exists constants $\gamma_\infty\in\RM$ and $C>0$ such that for each $N\in\NM$ we have
\begin{equation}\label{e:result2}
\left\|\widetilde{\psi}(\cdot,t)-\frac{1}{N}\gamma_\infty\right\|_{H^K_{\rm per}(0,N)}\leq C E_0(1+t)^{-1/4}
\end{equation}
for all $t\geq 0$.
\end{theorem}

\begin{remark}
Using the methods in \cite{JNRZ_13_1,JZ_11_1}, the results in Theorem \ref{T:main} can easily be extended to 
establish uniform (in $N$) decay rates of perturbations in $L^p_{\rm per}(0,N)$ for any $2\leq p\leq \infty$ provided the initial perturbations are again sufficiently
small in $L^1_{\rm per}(0,N)\cap H^K_{\rm per}(0,N)$.  For simplicity, however, and to establish proof of concept, in this work we concentrate on the $L^2$-based theory only.
\end{remark}

The key idea to the proof of Theorem \ref{T:main} is to use the stability theory of periodic waves of reaction diffusion equations to \emph{localized perturbations},
specifically those techniques developed in \cite{JNRZ_13_1,JZ_11_1}, as a guide for how to uniformly control the dynamics of subharmonic perturbations for large $N$.  
Indeed, observe the decay rates guaranteed in Theorem \ref{T:main} are precisely those predicted by considering the dynamics of such periodic wave trains
to localized perturbations: see \cite{JNRZ_Invent,JNRZ_13_1,JNRZ_13_2,JZ_11_1}, for example.  
Formally, this should not be too surprising since, up to appropriate translations, a sequence of $N$-periodic functions may converge (locally) as $N\to\infty$ to functions
in $L^2(\RM)$.   

We make the above intuition precise by first following the methodology recently developed in \cite{HJP_1}  in order to provide a delicate decomposition of the semigroup $e^{\mathcal{L}[\phi]t}$ acting on the 
the space $L^2_{\rm per}(0,N)$ with $N\in\NM$.  This decomposition is accomplished by adapting the linear theory for localized perturbations developed in \cite{JNRZ_13_1,JZ_11_1}
to the subharmonic context in order to uniformly handle the accumulation of Bloch eigenvalues near the origin as $N\to\infty$.  
Furthermore, our linear decomposition, which will be reviewed in Section \ref{S:lin_stab} below, 
not only recovers the exponential decay rates exhibited in Proposition \ref{P:sub_stab},
but they  also provide the uniform (in $N$) rates of decay in Theorem \ref{T:main}.  As we will see, this linear analysis predicts that if $u(x,t)$ is a solution
of \eqref{e:RDE_trav} which is initially close to $\phi$ in $L^2_{\rm per}(0,N)$ then there exists a (small) space-time dependent, $N$-periodic (in $x$) modulation function $\widetilde{\psi}(x,t)$
such that $u(x,t)$ essentially behaves for large time like
\[
u(x,t)\approx\phi(x)+\widetilde{\psi}(x,t)\phi'(x)\approx\phi\left(x+\widetilde{\psi}(x,t)\right)
\]
giving a refined insight into the long-time local dynamics near $\phi$ beyond the more standard asymptotic stability (with asymptotic phase) as in Proposition \ref{P:sub_stab}.
Motivated by this initial linear analysis, we then build a nonlinear iteration scheme for subharmonic perturbations which incorporates phase modulation functions
which depend on \emph{both space and time} in order to complete the proof of Theorem \ref{T:main}.  
The requirement that the modulation functions are spatially dependent is necessary for our method,
and is fundamentally different than the methodology used in the proof of Proposition \ref{P:sub_stab}.  In particular, to the authors' knowledge, this work
is the first to consider spatially dependent modulation functions in the context of periodic perturbations.  Furthermore,
Theorem \ref{T:main} is the first result to obtain stability results for periodic waves to subharmonic perturbations that are uniform in the
period of the perturbation.

\begin{remark}
As indicated above, the strategy for proving our subharmonic results follows the stability analyses \cite{JNRZ_13_1,JZ_11_1} for localized perturbations of periodic
wave trains in reaction diffusion systems.  In the localized case, the origin is always a part of the essential spectrum of the linearized operator, leading one 
to introduce space-time dependent modulation functions.  In the subharmonic case, however, the origin is an isolated simple eigenvalue for each fixed $N\in\NM$,
and using time-dependent modulations only leads to results such as Proposition \ref{P:sub_stab}.  In order to achieve the proof
of Theorem \ref{T:main}, we will rely on a combination of these approaches, using an $N$-dependent time-modulation function to account
for the isolated eigenvalue at the origin, while simultaneously using a space-time modulation to account for the accumulation of spectrum
near the origin as $N\to\infty$.  
\end{remark}

Next, we point out an important corollary of Theorem \ref{T:main}.  Particularly,
since the decay rates in Theorem \ref{T:main} are sufficiently fast we can obtain the following result accounting for only time-dependent modulations yet offering
slower uniform decay rates.  Note that while the result uses only time-dependent modulations, the proof requires the use of space-time dependent modulation functions.

\begin{corollary}\label{C:main}
Under the hypotheses of Theorem \ref{T:main}, there exists an $\eps>0$ and a constant $C>0$ such that, for each $N\in\NM$, whenever $u_0\in  L^1_{\rm per}(0,N)\cap H^{K}_{\rm per}(0,N)$ and
$E_0<\eps$, there exists a function $\gamma(t)$ satisfying $\gamma(0)=0$ such that 
the solution $u$ of \eqref{e:RDE_trav} with initial data $u(0)=u_0$ exists globally in time and satisfies
\begin{equation}\label{e:result3}
\left\|u\left(\cdot-\frac{1}{N}\gamma(t),t\right)-\phi\right\|_{H^K_{\rm per}(0,N)}\leq C E_0(1+t)^{-1/4}.
\end{equation}
for all $t>0$.   Further, the time-dependent modulation function $\gamma(t)$ satisfies
\[
\left|\gamma_t(t)\right|\leq CE_0(1+t)^{-3/2}
\]
and hence, in particular, there exists a $\gamma_\infty\in\RM$ 
\[
\left|\gamma(t)-\gamma_\infty\right|\leq CE_0(1+t)^{-1/2}
\]
for all $t>0$.  In particular,
\[
\left\|u\left(\cdot,t\right)-\phi\left(\cdot+\frac{1}{N}\gamma_\infty\right)\right\|_{H^K_{\rm per}(0,N)}\leq CE_0 (1+t)^{-1/4}.
\]
for all $t>0$.
\end{corollary}

\begin{remark}
Comparing Corollary \ref{C:main} with Proposition \ref{P:sub_stab}, we see that one necessarily has the relationship
$\sigma_\infty(N) = \frac{1}{N}\gamma_\infty$, establishing a direct correspondence between the ($N$-dependent) asymptotic phase shifts.  Further,
we note the $\gamma_\infty$ is the same in both Theorem \ref{T:main} and Corollary \ref{C:main}.
\end{remark}

Our last result combines the results of Corollary \ref{C:main} with Proposition \ref{P:sub_stab} in order to obtain a nonlinear stability
result allowing a uniform (in $N$) size of initial perturbations with (eventual) exponential rates of decay.

\begin{corollary}\label{C:min_thm}
Under the hypotheses of Theorem \ref{T:main}, there exists an $\eps>0$ and a constant $C>0$ such that, for each $N\in\NM$ and $\delta\in (0,\delta_N)$,
with $\delta_N$ as in \eqref{spec_gap},
whenever $u_0\in L^1_{\rm per}(0,N)\cap H^K_{\rm per}(0,N)$ with $E_0<\eps$ there exists a $T_\delta>0$ and a constant $M_\delta>0$ such that
\[
\left\|u(\cdot,t)-\phi\left(\cdot+\frac{1}{N}\gamma_\infty\right)\right\|_{H^1_{\rm per}(0,N)}\leq 
		\left\{\begin{aligned}
					&CE_0(1+t)^{-1/4},~~{\rm for}~~0<t\leq T_\delta\\
					&M_\delta E_0 e^{-\delta t},~~{\rm for}~~t>T_\delta.
					\end{aligned}\right.
\]
\end{corollary}

The above corollary has a few important features to highlight.  First, we emphasize that $\eps$, the size of the initial perturbation above, is independent
of both $N$ and the choice $\delta\in(0,\delta_N)$.  In particular, this establishes a uniform size on the domain of attraction for perturbations
to (eventually) exhibit exponential decay.  This is in stark contrast to Proposition \ref{P:sub_stab} which requires $\eps_\delta\to 0$ as $\delta\to 0$.
Secondly, we note that the length of time one must wait to observe exponential decay, quantified by $T_\delta$ above, necessarily satisfies
$T_\delta\to \infty$ as $\delta\to 0$; hence, it is not uniform in $N$.  Nevertheless, Corollary \ref{C:min_thm} upgrades the long-time
behavior of Proposition \ref{P:sub_stab} allowing for a uniform size of initial perturbations.  
Interestingly, Corollary \ref{C:min_thm} can be easily seen, at least at the linear level, directly from our forthcoming 
decomposition of the semigroup $e^{\mathcal{L}[\phi]t}$: see Remark \ref{R:riemann_sum1} in Section \ref{S:lin_stab} below.

\

The online of the paper is as follows.  In Section \ref{S:prelim} we review several preliminary results, including a review in Section \ref{S:bloch} of Floquet-Bloch theory in the context of $N$-periodic
function spaces.  This will provide us with a characterization of $N$-periodic eigenvalues of the $1$-periodic coefficient differential operator $\mathcal{L}[\phi]$ 
in terms of the associated Bloch operators.  We further collect several properties of the Bloch operators and their associated semigroups.
In Section \ref{S:spec_semigrp}, we establish basic decay properties of the Bloch semigroups arising as a result of the diffusive spectral stability assumption.
In Section \ref{S:lin_stab}, we establish our key linear estimates by providing a delicate decomposition
of the semigroup $e^{\mathcal{L}[\phi]t}$ acting on $L^2_{\rm per}(0,N)$, which allows us to identify polynomial decay rates on the linear evolution which
are \emph{uniform in $N$}: see Proposition \ref{P:lin_est}.  These linear estimates form the backbone for our nonlinear analysis, which is detailed
in Section \ref{S:nlin_stab}.  In Section \ref{S:nlin_decomp}, we use intuition gained from the linear estimates of Section \ref{S:lin_stab} to introduce
an appropriate nonlinear decomposition of a small $L^2_{\rm per}(0,N)$ neighborhood of the underlying diffusively stable $1$-periodic wave $\phi$,
and we develop appropriate perturbation equations satisfied by the corresponding perturbation and modulation functions.  In Section \ref{S:nlin_iteration},
we apply a nonlinear iteration scheme to the system of perturbation equations obtained in Section \ref{S:nlin_decomp} and present the 
proofs of Theorem \ref{T:main} and its corollaries stated above.  Finally, a proof of some technical results from Section \ref{S:lin_stab} are provided in an Appendix.
\\

\noindent
{\bf Acknowledgments:} The work of MAJ was partially funded by the NSF under grant DMS-16-14785, as well the Simons Foundation Collaboration grant number 714021.
The authors are also grateful to the referees for their many helpful suggestions.  Finally, we thank Prof. Guido Schneider for initial discussions regarding 
Corollary \ref{C:min_thm}.

\section{Preliminaries}\label{S:prelim}

In this section, we review several preliminary results.  First, to aid in our description of the spectrum of the linearization $\mathcal{L}[\phi]$, we review general
results from Floquet-Bloch theory as applied to subharmonic perturbations.  From this, we 
establish some elementary semigroup estimates for the associated Bloch operators.  Throughout the remainder of the paper, for notational convenience, we set for each $N\in\NM$
and $p\geq 1$
\[
L^p_N:=L^p_{\rm per}(0,N).
\]

\subsection{Floquet Bloch Theory for Subharmonic Perturbations}\label{S:bloch}

Motivated by Floquet-Bloch theory for linear differential operators with periodic coefficients acting on $L^2(\RM)$ (see \cite{G93,JNRZ_Invent,RS4}, for example),
we review a modification of this theory (restricted to the present reaction-diffusion context) for the study of subharmonic perturbations\footnote{See also
\cite{HJP_1} for more information regarding this subharmonic extension.}.

Suppose that $\phi$ is a  $1$-periodic stationary solution of \eqref{e:RDE_trav}, and consider the linearized operator $\mathcal{L}[\phi]$.
Since the coefficients of $\mathcal{L}[\phi]$ are $1$-periodic, Floquet theory implies that for each $\lambda\in\CM$ any non-trivial solution
of the ordinary differential equation
\[
\mathcal{L}[\phi]v=\lambda v
\]
cannot be integrable on $\RM$ and that, at best, they can be bounded functions of the form
\begin{equation}\label{e:floquet_form}
v(x)=e^{i\xi x}w(x)
\end{equation}
for some $\xi\in[-\pi,\pi)$ and non-trivial function $w\in L^2_{\rm per}(0,1)$. 
For a given $N\in\NM$, setting
\[
\Omega_N:=\left\{\xi\in[-\pi,\pi):e^{i\xi N}=1\right\}
\]
we see from \eqref{e:floquet_form} that the perturbation $v$ satisfies $N$-periodic boundary conditions if and only if $\xi\in\Omega_N$.  In particular,
it can be shown that $\lambda\in\CM$ belongs to the $L^2_N$-spectrum of $\mathcal{L}[\phi]$ if and only if there exits a $\xi\in\Omega_N$
and a non-trivial $w\in L^2_{\rm per}(0,1)$ such that
\[
\lambda w=e^{-i\xi x}\mathcal{L}[\phi]e^{i\xi x}w=:\mathcal{L}_\xi[\phi]w.
\]
The operators $\mathcal{L}_\xi[\phi]$ are known as the Bloch operators associated to $\mathcal{L}[\phi]$, and the parameter $\xi$ is referred to as the Bloch frequency.  Note
that each $\mathcal{L}_\xi[\phi]$ acts on $L^2_{\rm per}(0,1)$ with densely defined and compactly embedded domain $H^1_{\rm per}(0,1)$, and hence their spectrum consists entirely
of isolated eigenvalues with finite algebraic multiplicities which, furthermore, depend continuously on $\xi$.  In fact, we have the spectral decomposition
\[
\sigma_{L^2_N}\left(\mathcal{L}[\phi]\right)=\bigcup_{\xi\in\Omega_N}\sigma_{L^2_{\rm per}(0,1)}\left(\mathcal{L}_\xi[\phi]\right).
\]
This characterizes the $N$-periodic spectrum of $\mathcal{L}[\phi]$ in terms of union of $1$-periodic eigenvalues for the Bloch operators $\{\mathcal{L}_\xi[\phi]\}_{\xi\in\Omega_N}$.

\begin{remark}
For definiteness, we note that the set $\Omega_N$ may be written explicitly when $N$ is even by 
\[
\Omega_N=\left\{\xi_j=\frac{2\pi j}{N}:j=-\frac{N}{2},~-\frac{N}{2}+1,\ldots,\frac{N}{2}-1\right\}
\]
and when $N$ is odd by
\[
\Omega_N=\left\{\xi_j=\frac{2\pi j}{N}:j=-\frac{N-1}{2},~-\frac{N-1}{2}+1,\ldots,\frac{N-1}{2}\right\}.
\]
In particular, observe that we have $0\in\Omega_N$ and $|\Omega_N|=N$ for all $N\in\NM$ and that, furthermore, $\Delta\xi_j:=\xi_j-\xi_{j-1}=\frac{2\pi}{N}$ for each
appropriate $j$.
\end{remark}

From the above, it is clearly desirable to have the ability to decompose arbitrary functions in $L^2_N$ into superpositions of functions of the form $e^{i\xi x}w(x)$ with 
$\xi\in\Omega_N$ and $w\in L^2_{\rm per}(0,1)$.  This is achieved by noting that a given $g\in L^2_N$ admits a Fourier series representation
\[
g(x)=\frac{1}{N}\sum_{m\in\ZM}e^{2\pi imx/N}\widehat{g}\left(2\pi m/N\right)
\]
where here $\widehat{g}$ denotes the Fourier transform of $g$ on the torus given by
\begin{equation}\label{e:fourier_def}
\widehat{g}(z):=\int_{-N/2}^{N/2} e^{-izy}g(y)dy.
\end{equation}
Together with the identity (valid for any $f$ for which the sum converges)
\[
\sum_{m\in\ZM}f\left(2\pi m/N\right)=\sum_{\xi\in\Omega_N}\sum_{\ell\in\ZM}f\left(\xi+2\pi\ell\right),
\]
it follows that $g$ may be represented as
\[
g(x)=\frac{1}{N}\sum_{\xi\in\Omega_N}\sum_{\ell\in\ZM}e^{i(\xi+2\pi\ell)x}\widehat{g}\left(\xi+2\pi\ell\right).
\]
In particular, defining for $\xi\in\Omega_N$ the $1$-periodic Bloch transform of a function $g\in L^2_N$ as
\[
\mathcal{B}_1(g)(\xi,x):=\sum_{\ell\in\ZM} e^{2\pi i\ell x}\widehat{g}(\xi+2\pi\ell),
\]
the above yields the inverse Bloch representation formula
\[
g(x)=\frac{1}{N}\sum_{\xi\in\Omega_N}e^{i\xi x}\mathcal{B}_1(g)(\xi,x),
\]
which is valid for all $g\in L^2_N$.  Note that the function $\mathcal{B}_1(g)(\xi,\cdot)$ is clearly $1$-periodic for each $\xi\in\Omega_N$,
and hence the above representation formula decomposes arbitrary $N$-periodic functions in the desired fashion.

Before proceeding, we note that, in fact, the $1$-periodic Bloch transform 
\[
\mathcal{B}_1:L^2_N\to \ell^2\left(\Omega_N:L^2_{\rm per}(0,1)\right)
\]
as defined above satisfies the subharmonic Parseval identity
\begin{equation}\label{e:parseval_per}
\left<f,g\right>_{L^2_N}=\frac{1}{N}\sum_{\xi\in\Omega_N}\left<\mathcal{B}_1(f)(\xi,\cdot),\mathcal{B}_1(g)(\xi,\cdot)\right>_{L^2(0,1)}
\end{equation}
valid for all $f,g\in L^2_N$.   In particular, this yields the useful identity
\[
\|g\|_{L^2_N}^2=\frac{1}{N}\sum_{\xi\in\Omega_N}\left\|\mathcal{B}_1(g)(\xi,\cdot)\right\|_{L^2(0,1)}^2
\]
valid for all $g\in L^2_N$, establishing that (up to normalization) $\mathcal{B}_1$ is an isometry.
Furthermore, we note that
\[
\mathcal{B}_1\left(\mathcal{L}[\phi]v\right)(\xi,x)=\left(\mathcal{L}_\xi[\phi]\mathcal{B}_1(v)(\xi,\cdot)\right)(x)~~{\rm and}~~
	\mathcal{L}[\phi]v(x)=\frac{1}{N}\sum_{\xi\in\Omega_N} e^{i\xi x}\mathcal{L}_\xi[\phi]\mathcal{B}_1(v)(\xi,x).
\]
and hence we may view the Bloch operators $\mathcal{L}_\xi[\phi]$ as operator valued symbols associated to $\mathcal{L}[\phi]$ under the action
of the $1$-periodic Bloch transform $\mathcal{B}_1$.  
Since the operator $\mathcal{L}[\phi]$ and its corresponding Bloch operators $\mathcal{L}_\xi[\phi]$ are clearly sectorial on $L^2_N$ and $L^2_{\rm per}(0,1)$, respectively,
they clearly generate analytic semigroups on their respective function spaces and, further, it is now straightforward to check that the associated semigroups satisfy
\begin{equation}\label{e:per_semigrp}
\mathcal{B}_1\left(e^{\mathcal{L}[\phi]t}v\right)(\xi,x)=\left(e^{\mathcal{L}_\xi[\phi]t}\mathcal{B}_1(v)(\xi,\cdot)\right)(x)~~{\rm and}~~
	e^{\mathcal{L}[\phi]t}v(x)=\frac{1}{N}\sum_{\xi\in\Omega_N} e^{i\xi x}e^{\mathcal{L}_\xi[\phi]t}\mathcal{B}_1(v)(\xi,x).
\end{equation}
Combined with \eqref{e:parseval_per}, this latter identity allows us to conclude information about the semigroup $e^{\mathcal{L}[\phi]t}$ 
acting on $L^2_N$ by synthesizing  (over $\xi\in\Omega_N$) information
about the Bloch semigroups $e^{\mathcal{L}_\xi[\phi]t}$ acting on $L^2_{\rm per}(0,1)$.  This decomposition is
key to our forthcoming linear analysis.

Finally, we end by recalling the following useful identity.

\begin{lemma}\label{L:per_factor_lemma}
Let $N\in\NM$.  If $f\in L^2_{\rm per}(0,1)$ and $g\in L^2_N$, then
\[
\mathcal{B}_1(fg)(\xi,x)=f(x)\mathcal{B}_1(g)(\xi,x).
\]
In particular, for such $f$ and $g$ we have the identity
\[
\left<f,g\right>_{L^2_N}=\left<f,\mathcal{B}_1(g)(0,\cdot)\right>_{L^2(0,1)}.
\]
\end{lemma}

The proof of Lemma \ref{L:per_factor_lemma} is straightforward and can be found in \cite{HJP_1}.

\subsection{Diffusive Spectral Stability \& Properties of Semigroups}\label{S:spec_semigrp}

With the above characterization of the $L^2_N$-spectrum of the linearized operator $\mathcal{L}[\phi]$ about a $1$-periodic stationary solution $\phi$ of \eqref{e:RDE_trav},
we can now provide some immediate consequences of the diffusive spectral stability assumption in Definition \ref{D:specstab}.  
Specifically, in our present subharmonic context we note that if $\phi$ is such  a diffusively spectrally stable standing solution of \eqref{e:RDE_trav}, 
then for each $N\in\NM$ there exists a $\delta_N>0$ such that
\[
\Re\left(\sigma_{L^2_N}\left(\mathcal{L}[\phi]\right)\setminus\{0\}\right)\leq-\delta_N,
\]
i.e. the non-zero $N$-periodic eigenvalues of $\mathcal{L}[\phi]$ are uniformly bounded away from the imaginary axis.  
In particular, by standard spectral perturbation theory, we immediately have that the following spectral properties hold.

\begin{lemma}[Spectral Preparation]\label{L:specprep}
Suppose that $\phi$ is a $1$-periodic stationary solution of \eqref{e:RDE_trav} which is diffusively spectrally stable.  Then the following properties hold.
\begin{itemize}
\item[(i)] For any fixed $\xi_0\in(0,\pi)$, there exists a constant $\delta_0>0$ such that
\[
\Re\left(\sigma\left(\mathcal{L}_\xi[\phi]\right)\right)<-\delta_0
\]
for all $\xi\in[-\pi,\pi)$ with $|\xi|>\xi_0$.
\item[(ii)] There exist positive constants $\xi_1$ and $\delta_1$ such that for any $|\xi|<\xi_1$, the spectrum of $\mathcal{L}_\xi[\phi]$ decomposes
into two disjoint subsets
\[
\sigma\left(\mathcal{L}_\xi[\phi]\right)=\sigma_-\left(\mathcal{L}_\xi[\phi]\right)\bigcup\sigma_0\left(\mathcal{L}_\xi[\phi]\right)
\]
with the following properties:
\begin{itemize}
\item[(a)] $\Re~\sigma_-\left(\mathcal{L}_\xi[\phi]\right)<-\delta_1$ and $\Re~\sigma_0\left(\mathcal{L}_\xi[\phi]\right)>-\delta_1$;
\item[(b)] the set $\sigma_0\left(\mathcal{L}_\xi[\phi]\right)$ consists of a single eigenvalue $\lambda_c(\xi)$ which is analytic in $\xi$
and expands as
\[
\lambda_c(\xi)=ia\xi-d\xi^2+\mathcal{O}(\xi^3)
\]
for $|\xi|\ll 1$ and some constants $a\in\RM$ and $d>0$;
\item[(c)] the eigenfunction associated to $\lambda_c(\xi)$ is analytic near $\xi=0$ and expands as
\[
\Phi_\xi(x)=\phi'(x)+\mathcal{O}(\xi)
\]
for $|\xi|\ll 1$.
\end{itemize}
\end{itemize}
\end{lemma}

The proof of (i) follows immediately from the properties (i) and (ii) in Definition \ref{D:specstab}, while the second part follows since $\lambda=0$
is a simple eigenvalue of the co-periodic operator $\mathcal{L}_0[\phi]$ and that the coefficients of $\mathcal{L}_\xi[\phi]$ clearly vary analytically
on $\xi$.  

With the above spectral preparation result in hand, we now record some key induced features of the associated semigroups.  These estimates are immediate consequences
of Lemma \ref{L:specprep} and the fact that the Bloch operators are clearly sectorial when acting on $L^2_{\rm per}(0,1)$.

\begin{proposition}\label{P:hfexp_decay_est}
Suppose that $\phi$ is a $1$-periodic stationary solution of \eqref{e:RDE_trav} which is diffusively spectrally stable.  Then the following properties hold.
\begin{itemize}
\item[(i)] For any fixed $\xi_0\in(0,\pi)$, there exist positive constants $C_0$ and $d_0$ such that
\[
\left\|e^{\mathcal{L}_\xi[\phi]t}f\right\|_{B(L^2_{\rm per}(0,1))}\leq C_0 e^{-d_0 t}
\]
valid for all $t\geq 0$ and all $\xi\in[-\pi,\pi)$ with $|\xi|>\xi_0$.
\item[(ii)] With $\xi_1$ chosen as in Lemma \ref{L:specprep}, there exist positive constants $C_1$ and $d_1$ such that for any $|\xi|<\xi_1$, if $\Pi(\xi)$ denotes the (rank-one)
spectral projection onto the eigenspace associated to $\lambda_c(\xi)$ given by Lemma \ref{L:specprep}(ii), then
\[
\left\|e^{\mathcal{L}_\xi[\phi]t}\left(1-\Pi(\xi)\right)\right\|_{B(L^2_{\rm per}(0,1))}\leq C_1 e^{-d_1 t}
\]
for all $t\geq 0$.
\end{itemize}
\end{proposition}

Coupled with an appropriate decomposition of $e^{\mathcal{L}[\phi]t}$, the above linear estimates form the core of our forthcoming linear analysis 
(which, in turn, forms the backbone of our nonlinear iteration scheme).

\section{Uniform Subharmonic Linear Estimates}\label{S:lin_stab}

We begin our analysis by obtaining decay rates on the semigroup $e^{\mathcal{L}[\phi]t}$ acting on classes of subharmonic perturbations in $L^2_N$
which are uniform in $N$.  This analysis is based on a delicate decomposition of the semigroup.  In particular, we use \eqref{e:per_semigrp} to study the action
of $e^{\mathcal{L}[\phi]t}$ on $L^2_N$ in terms of associated Bloch operators, which is accomplished by separating the semigroup into appropriate critical frequency and non-critical frequency components.
Note that, due to Lemma \ref{L:specprep} we expect the ``critical frequency" component to be dominated by the translational mode $\phi'$.  This decomposition
was recently carried out in detail (in a related context) in \cite{HJP_1}, and for completeness we review it here.  Note the decomposition is heavily
motivated by the corresponding decomposition used in the case of localized perturbations: see \cite{JNRZ_Invent,JNRZ_13_1}.

To begin, let $\xi_1\in(0,\pi)$ be defined as in Lemma \ref{L:specprep} and let $\rho$ be a smooth cutoff function satisfying $\rho(\xi)=1$ 
for $|\xi|<\frac{\xi_1}{2}$ and $\rho(\xi)=0$ for $|\xi|>\xi_1$.  For a given $v\in L^2_ N$, we use \eqref{e:per_semigrp} to decompose $e^{\mathcal{L}[\phi]t}$ into low-frequency
and high-frequency components as
\begin{equation}\label{e:lf_hf_decomp}
\begin{aligned}
e^{\mathcal{L}[\phi]t}v(x)&=\frac{1}{N}\sum_{\xi\in\Omega_N}\rho(\xi)e^{i\xi x}e^{\mathcal{L}_\xi[\phi]t}\mathcal{B}_1(v)(\xi,x)+
				\frac{1}{N}\sum_{\xi\in\Omega_N}\left(1-\rho(\xi)\right)e^{i\xi x}e^{\mathcal{L}_\xi[\phi]t}\mathcal{B}_1(v)(\xi,x)\\
				&=S_{lf,N}(t)v(x) + S_{hf,N}v(x).
\end{aligned}
\end{equation}
Using Proposition  \ref{P:hfexp_decay_est} and the subharmonic Parseval identity \ref{e:parseval_per}, it follows that there exist constants $C,\eta>0$, both independent of $N$, such that
\begin{align*}
\left\|S_{hf,N}(t)v\right\|_{L^2_N}^2&=\frac{1}{N}\sum_{\xi\in\Omega_N}\left\|(1-\rho(\xi))e^{\mathcal{L}_\xi[\phi]t}\mathcal{B}_1(v)(\xi,\cdot)\right\|_{L^2(0,1)}^2\\
	&\leq\frac{1}{N}\sum_{\xi\in\Omega_N}(1-\rho(\xi))^2\left\|e^{\mathcal{L}_\xi[\phi]t}\right\|_{B(L^2(0,1))}^2\left\|\mathcal{B}_1(v)(\xi,\cdot)\right\|_{L^2(0,1)}^2\\
	&\leq Ce^{-2\eta t}\left(\frac{1}{N}\sum_{\xi\in\Omega_N}\left\|\mathcal{B}_1(v)(\xi,\cdot)\right\|_{L^2(0,1)}^2\right),
\end{align*}
which, again using Parseval's identity \eqref{e:parseval_per}, yields the exponential decay estimate
\begin{equation}\label{e:exp_est1}
\left\|S_{hf,N}(t)v\right\|_{L^2_N}\leq Ce^{-\eta t}\|v\|_{L^2_N}.
\end{equation}

For the low-frequency component, for each $|\xi|<\xi_1$ define the rank-one spectral projection onto the critical mode of $\mathcal{L}_\xi[\phi]$ by
\begin{equation}\label{e:spec_proj}
\left\{\begin{aligned}
&\Pi(\xi):L^2_{\rm per}(0,1)\to{\rm ker}\left(\mathcal{L}_\xi[\phi]-\lambda_c(\xi)I\right)\\
&\Pi(\xi)g(x)=\left<\widetilde{\Phi}_\xi,g\right>_{L^2(0,1)}\Phi_\xi(x)
\end{aligned}\right.
\end{equation}
where here $\widetilde{\Phi}_\xi$ denotes the element of the kernel of the adjoint $\mathcal{L}_\xi[\phi]^\dag-\overline{\lambda_c(\xi)}I$ satisfying the normalization
condition $\left<\widetilde{\Phi}_\xi,\Phi_\xi\right>_{L^2(0,1)}=1$.  The low-frequency operator $S_{lf,N}$ can thus be further decomposed into the contribution from the critical
mode and the contribution from low-frequency spectrum bounded away from $\lambda=0$ via
\begin{equation}\label{e:lf_c_decomp}
\begin{aligned}
S_{lf,N}(t)v(x)&=\frac{1}{N}\sum_{\xi\in\Omega_N}\rho(\xi)e^{i\xi x}e^{\mathcal{L}_\xi[\phi]t}\Pi(\xi)\mathcal{B}_1(v)(\xi,x)
			+\frac{1}{N}\sum_{\xi\in\Omega_N}\rho(\xi)e^{i\xi x}e^{\mathcal{L}_\xi[\phi]t}\left(1-\Pi(\xi)\right)\mathcal{B}_1(v)(\xi,x)\\
			&=:S_{c,N}v(x)+\widetilde{S}_{lf,N}(t)v(x).
\end{aligned}
\end{equation}
As with the exponential estimate \eqref{e:exp_est1}, Proposition \ref{P:hfexp_decay_est} implies, by possibly choosing $\eta>0$ smaller,	 that there exists a constant $C>0$ independent
of $N$ such that
\begin{equation}\label{e:exp_est2}
\left\|\widetilde{S}_{lf,N}(t)v\right\|_{L^2_N}\leq Ce^{-\eta t}\|v\|_{L^2_N}.
\end{equation}
For the critical component $S_{c,N}$, note by Lemma \ref{L:specprep}(ii) that we can write
\begin{align*}
S_{c,N}(t)v(x)&=\frac{1}{N}e^{\mathcal{L}_0[\phi]t}\Pi(0)\mathcal{B}_1(v)(0,x)
	+\frac{1}{N}\sum_{\xi\in\Omega_N\setminus\{0\}}\rho(\xi)e^{i\xi x}e^{\mathcal{L}_\xi[\phi]t}\Pi(\xi)\mathcal{B}_1(v)(\xi,x)\\
&=\frac{1}{N}\phi'(x)\left<\widetilde{\Phi}_0,\mathcal{B}_1(v)(0,\cdot)\right>_{L^2(0,1)}
	+\frac{1}{N}\sum_{\xi\in\Omega_N\setminus\{0\}}\rho(\xi)e^{i\xi x}e^{\lambda_c(\xi)t}\Phi_\xi(x)\left<\widetilde{\Phi}_\xi,\mathcal{B}_1(v)(\xi,\cdot)\right>_{L^2(0,1)}.
\end{align*}
and hence, recalling Lemma \ref{L:per_factor_lemma} and expanding $\Phi_\xi$,
\begin{align*}
S_{c,N}(t)v(x)&=\frac{1}{N}\phi'(x)\left<\widetilde{\Phi}_0,v\right>_{L^2_N}
	+\phi'(x)\frac{1}{N}\sum_{\xi\in\Omega_N\setminus\{0\}}\rho(\xi)e^{i\xi x}e^{\lambda_c(\xi)t}\left<\widetilde{\Phi}_\xi,\mathcal{B}_1(v)(\xi,\cdot)\right>_{L^2(0,1)}\\
&\quad + \frac{1}{N}\sum_{\xi\in\Omega_N\setminus\{0\}}\rho(\xi)e^{i\xi x}(i\xi)e^{\lambda_c(\xi)t}
	\left(\frac{\widetilde{\Phi}_\xi(x)-\phi'(x)}{i\xi}\right)\left<\widetilde{\Phi}_\xi,\mathcal{B}_1(v)(\xi,\cdot)\right>_{L^2(0,1)}\\
&=:\frac{1}{N}\phi'(x)\left<\widetilde{\Phi}_0,v\right>_{L^2_N}+\phi'(x)s_{p,N}(t)v(x)+\widetilde{S}_{c,N}(t)v(x).
\end{align*}

Taken together, it follows that the linear solution operator $e^{\mathcal{L}[\phi]t}$ can be decomposed as
\begin{equation}\label{e:lin_decomp1}
e^{\mathcal{L}[\phi]t}v(x)=\frac{1}{N}\phi'(x)\left<\widetilde{\Phi}_0,v\right>_{L^2_N}+\phi'(x)s_{p,N}(t)v(x)+\widetilde{S}_N(t)v(x)
\end{equation}
where 
\begin{equation}\label{e:sp}
s_{p,N}(t)v(x)=\frac{1}{N}\sum_{\xi\in\Omega_N\setminus\{0\}}\rho(\xi)e^{i\xi x}e^{\lambda_c(\xi)t}\left<\widetilde{\Phi}_\xi,\mathcal{B}_1(v)(\xi,\cdot)\right>_{L^2(0,1)}
\end{equation}
and
\[
\widetilde{S}_N(t)v(x)=S_{hf,N}(t)v(x)+\widetilde{S}_{lf,N}(t)v(x)+\widetilde{S}_{c,N}(t)v(x).
\]
Equipped with the above, we can establish our main set of linear estimates.  

\begin{proposition}[Linear Estimates]\label{P:lin_est}
Suppose that $\phi$ is a $1$-periodic stationary solution of \eqref{e:RDE_trav} which is diffusively spectrally stable.  
Given any $M\in\NM$, there exists a constant $C>0$ such that for all $t\geq 0$, $N\in\NM$ 
and all $0\leq l,m\leq M$ we have
\[
\left\|\partial_x^l\partial_t^m s_{p,N}(t) v\right\|_{L^2_N}\leq C (1+t)^{-1/4-(l+m)/2}\|v\|_{L^1_N}
\]
Furthermore, there exists constants $C,\eta>0$ such that for all $t\geq 0$ and  $N\in\NM$ 
we have
\[
\left\|\widetilde{S}_N(t) v\right\|_{L^2_N}\leq C\left( (1+t)^{-3/4}\|v\|_{L^1_N}+e^{-\eta t}\left\| v\right\|_{L^2_N}\right).
\]
\end{proposition}

\begin{remark}
While the bounds above on the derivatives of $s_{p,N}(t)$ are largely unmotivated by our linear analysis, they will be essential in our
forthcoming nonlinear theory.
\end{remark}

\begin{proof}
First observe that, by definition of $\mathcal{B}_1$, we have
\begin{align*}
\left<\widetilde{\Phi}_\xi,\mathcal{B}_1(v)(\xi,\cdot)\right>_{L^2(0,1)}&=\int_0^1\overline{\widetilde{\Phi}_\xi(x)}\sum_{\ell\in\ZM}e^{2\pi i\ell x}\widehat{v}(\xi+2\pi\ell)dx\\
&=\sum_{\ell\in\ZM}\widehat{v}(\xi+2\pi\ell)\int_0^1\overline{\widetilde{\Phi}_\xi(x)} e^{2\pi i\ell x}dx\\
&=\sum_{\ell\in\ZM}\widehat{v}(\xi+2\pi\ell)\overline{\widehat{\widetilde{\Phi}_\xi}(2\pi\ell)}
\end{align*}
and hence, using the fact that \eqref{e:fourier_def} implies $\|\widehat{v}\|_{L^\infty(\RM)}\leq\|v\|_{L^1_N}$ along with Cauchy-Schwartz, it follows that
\begin{align*}
\rho(\xi)\left|\left<\widetilde{\Phi}_\xi,\mathcal{B}_1(v)(\xi,\cdot)\right>_{L^2(0,1)}\right|^2&\leq\rho(\xi)\|v\|_{L^1_N}^2
				\left(\sum_{\ell\in\ZM}(1+|\ell|^2)^{1/2}\left|\overline{\widehat{\widetilde{\Phi}_\xi}(2\pi\ell)}\right|(1+|\ell|^2)^{-1/2}\right)^2\\
&\leq C\|v\|_{L^1_N}^2\sup_{\xi\in[-\pi,\pi)}\left(\rho(\xi)\left\|\widetilde{\Phi}_\xi\right\|_{H^1_{\rm per}(0,1)}^2\right).
\end{align*}
valid for all $\xi\in\Omega_N$.  Using Lemma \ref{L:specprep}, it follows by Parseval's identity \eqref{e:parseval_per} that there exists constants $C,d>0$, independent of $N$,
such that
\begin{align*}
\left\|\partial_x^l\partial_t^m s_{p,N}(t)v\right\|_{L^2_N}^2 &= \frac{1}{N}\sum_{\xi\in\Omega_N}\left\|\rho(\xi)(i\xi)^l\left(\lambda_c(\xi)\right)^me^{\lambda_c(\xi)t}
	\left<\widetilde{\Phi}_\xi,\mathcal{B}_1(v)(\xi,\cdot)\right>_{L^2(0,1)}\right\|_{L^2(0,1)}^2\\
&\leq C\|v\|_{L^1_N}^2\left(\frac{1}{N}\sum_{\xi\in\Omega_N}|\xi|^{2(l+m)}e^{-2d\xi^2 t}\right).
\end{align*}
By similar considerations, we find that
\[
\left\|\widetilde{S}_{N}(t)v\right\|_{L^2_N}^2\leq Ce^{-2\eta t}\|v\|_{L^2_N}^2+C\|v\|_{L^1_N}^2\left(\frac{1}{N}\sum_{\xi\in\Omega_N}|\xi|^2 e^{-2d\xi^2t}\right).
\]

It remains to provide uniform in $N$ decay rates on the finite sums
\begin{equation}\label{sum1}
\frac{1}{N}\sum_{\xi\in\Omega_N}|\xi|^{2(l+m)}e^{-2d\xi^2 t}~~{\rm and}~~\frac{1}{N}\sum_{\xi\in\Omega_N}|\xi|^2 e^{-2d\xi^2t}.
\end{equation}
To gain some intuition on how to uniformly bound these sums, notice that they can be interpreted as Riemann sum approximations (up to a harmless rescaling) of the integrals 
\begin{equation}\label{int1}
    \int_{-\pi}^{\pi} \xi^{2(\ell+m)}e^{-2d\xi^2 t}d\xi, \qquad \int_{-\pi}^{\pi} \xi^{2}e^{-2d\xi^2 t}d\xi,
\end{equation}
which, through an elementary scaling argument, exhibit $(1+t)^{-1/2-(\ell+m)}$ and $(1+t)^{-3/2}$ decay for large time, respectively.  The proof that the Riemann sums are uniformly controlled by these decay rates is provided in Lemma \ref{L:sum_poly_bd} in the Appendix, which completes the proof.
\end{proof}

\begin{remark}\label{R:riemann_sum1}
The result of Corollary \ref{C:min_thm} can be seen from the above analysis, at least at the linear level.  Indeed, 
following the methods in \cite[Section 5]{HJP_1} one sees that, for large $N$, the sums in \eqref{sum1} are good approximations of the 
respective integrals in \eqref{int1} for times up to $t=\mathcal{O}(N^2)$, corresponding to an observed polynomial decay of perturbations
on such a timescale.  For larger times, however, the exponential nature of the summands dominate and the sums decay monotonically to
zero at exponential rates, corresponding to an exponential decay of perturbations on these longer timescales.
\end{remark}

Before continuing to our nonlinear analysis, we pause to interpret the above results.  Suppose that  $\phi$ is a $1$-periodic diffusively spectrally stable stationary
solution of \eqref{e:RDE_trav}, and let $u(x,t)$ be a solution of \eqref{e:RDE_trav} with initial data $u(x,0)=\phi(x)+\eps v(x)$ with $\eps\ll 1$ and $v\in L^1_N\cap L^2_N$.  
From Proposition \ref{P:lin_est}, it follows that one may expect that the solution $u$ behaves for large time like
\begin{equation}\label{e:intuition}
\begin{aligned}
u(x,t)&\approx \phi(x)+\eps e^{\mathcal{L}[\phi]t}v(x)\\
&\approx \phi(x)+\eps \phi'(x)\left(\frac{1}{N}\left<\widetilde{\Phi}_0,v\right>_{L^2_N}+s_{p,N}(t)v(x)\right)\\
&\approx \phi\left(x+\eps\left(\frac{1}{N}\left<\widetilde{\Phi}_0,v\right>_{L^2_N}+s_{p,N}(t)v(x)\right)\right),
\end{aligned}
\end{equation}
which is a space-time dependent phase modulation of the underlying periodic wave $\phi$.  More precisely, note the phase modulation naturally
decomposes into two parts: a spatially independent component coming from the projection of the perturbation onto the translational eigenvalue at the origin, and a space-time dependent
component accounting for the dynamics associated to the accumulation of Bloch eigenvalues near the origin for large $N$.
In the next section, we use this linear intuition to develop a nonlinear iteration scheme and complete the proof of Theorem \ref{T:main} and its corollaries.

\section{Uniform Nonlinear Asymptotic Stability}\label{S:nlin_stab}

In this section, we use the decomposition of the linearized solution operator $e^{\mathcal{L}[\phi]t}$ and the associated linear estimates
in Proposition \ref{P:lin_est} to develop a nonlinear iteration scheme to complete the proof of Theorem \ref{T:main}.  
As discussed at the end of Section \ref{S:lin_stab}, the linear estimates in Proposition \ref{P:lin_est} suggest that if $\phi$ is a $1$-periodic
diffusively spectrally stable stationary solution of \eqref{e:RDE_trav}, then $N$-periodic perturbations of $\phi$ should, for large time, behave
essentially like space-time modulated version of $\phi$.  This suggests a nonlinear decomposition of $N$-periodic perturbations of $\phi$,
which we develop in Section \ref{S:nlin_decomp} below.   With this decomposition in hand, the proof of Theorem \ref{T:main} will be completed
in Section \ref{S:nlin_iteration} through an appropriate nonlinear iteration scheme.

\subsection{Nonlinear Decomposition and Perturbation Equations}\label{S:nlin_decomp}

Suppose $\phi$ is a $1$-periodic diffusively spectrally stable stationary solution of \eqref{e:RDE_trav}.
Motivated by the work in the previous section, we introduce a decomposition of nonlinear perturbations of the background wave $\phi$
which accounts for the critical phase-shift contribution $s_{p,N}(t)$ of the linear operator.  

Motivated by \eqref{e:intuition}, we begin by letting 
$\widetilde{u}(x,t)$ be a solution
of \eqref{e:RDE_trav} and define a spatially modulated function
\begin{equation}\label{e:mod}
u(x,t):=\widetilde{u}\left(x-\frac{1}{N}\gamma(t)-\psi(x,t),t\right)
\end{equation}
where both $\gamma:\RM_+\to\RM$ and $\psi:\RM\times\RM_+\to\RM$ are functions to be determined later.
Taking $\widetilde{u}$ to be initially close to $\phi$ in some sense, we attempt to decompose $u$ as
\begin{equation}\label{e:nlin_residual}
u(x,t)=\phi(x)+v(x,t),
\end{equation}
where here $v$ denotes a nonlinear perturbation.  Note that the form of the modulation in \eqref{e:mod} is a combination of (i) a time-dependent modulation, as one would
utilize in the proof of Proposition \ref{P:sub_stab}, and (ii) a space-time dependent modulation, as is used in the study of localized perturbations of 
periodic waves \cite{JNRZ_13_1,JZ_11_1}.  Consequently, the forthcoming nonlinear analysis is essentially a mixture of these two approaches.

As a preliminary step, we derive equations that must be satisfied by the perturbation $v$ and the modulation functions $\gamma$ and $\psi$.
To this end, we note that in \cite{JNRZ_13_1,JZ_11_1} it is shown through elementary, but tedious, manipulations that if $u(x,t)$ is as above then the triple $(v,\gamma,\psi)$ satisfies
\begin{equation}\label{e:pert1}
    (k\d_t - k\mathcal{L}[\phi])\left(v+\frac{1}{N}\phi'\gamma +\phi'\psi\right) = k\widetilde{\mathcal{N}}, \quad\text{where}~ 
    		k\widetilde{\mathcal{N}} := \widetilde{\mathcal{Q}} + k\widetilde{\mathcal{R}}_x + k\widetilde{\mathcal{S}}_t + \widetilde{\mathcal{T}},
\end{equation}
with
\[
    \widetilde{\mathcal{Q}}:= f(\phi+v) - f(\phi) - Df(\phi)v,\quad
    \widetilde{\mathcal{R}}:= -\psi_t v - \frac{1}{N}\gamma_t v + k \left(\frac{\psi_x}{1-\psi_x} v_x\right) + k \left(\frac{\psi_x^2}{1-\psi_x} \phi'\right),
\]
and
\[
    \widetilde{\mathcal{S}}:= \psi_x v,\quad
    \widetilde{\mathcal{T}}:= -\psi_x\left[f(\phi+v) - f(\phi)\right].
\]
Rearranging slightly as in \cite{DS_18} to remove temporal derivatives of the perturbation $v$ in present in $\widetilde{\mathcal{N}}$ in \eqref{e:pert1} yields the following.

\begin{lemma}\label{L:nlin_pert}
The nonlinear residual $v$ defined in \eqref{e:nlin_residual} and modulation functions $\gamma$ and $\psi$ in \eqref{e:mod} satisfy
\begin{equation}\label{e:pert2}
    (k\d_t - k\mathcal{L}[\phi])\left((1-\psi_x)v+\frac{1}{N}\phi'\gamma +\phi'\psi\right) = k\mathcal{N}, \quad\text{where}~ k\mathcal{N} = \mathcal{Q} + k\mathcal{R}_x,
\end{equation}
where here
\begin{equation}
    \mathcal{Q} = (1-\psi_x)\left[f(\phi+v) - f(\phi) - Df(\phi)v\right],
\end{equation}
and
\begin{equation}
    \mathcal{R} = -\psi_t v - \frac{1}{N}\gamma_t v + c\psi_x v + k (\psi_x v)_x + k \left(\frac{\psi_x}{1-\psi_x} v_x\right) + k \left(\frac{\psi_x^2}{1-\psi_x} \phi'\right).
\end{equation}
\end{lemma}

Our goal is to now obtain a closed nonlinear iteration scheme by integrating \eqref{e:pert2} and exploiting the decomposition of the linear solution operator
$e^{\mathcal{L}[\phi]t}$ provided in \eqref{e:lin_decomp1}.  To motivate this, we first provide an informal description of how to determine the modulation
functions $\gamma$ and $\psi$ to separate out the principle nonlinear behavior.  Using Duhamel's formula, we can write \eqref{e:pert2} as the implicit
integral equation
\[
    (1-\psi_x(x,t))v(x,t)+\frac{1}{N}\phi'(x)\gamma(t) +\phi'(x)\psi(x,t) = e^{\mathcal{L}[\phi]t}
    v(x,0) + \int_0^t e^{\mathcal{L}[\phi](t-s)}\mathcal{N}(x,s)ds.
\]
with initial data $\gamma(0)=0$, $\psi(\cdot,0)=0$ and $v(\cdot,0)=\widetilde{u}(\cdot,0)-\phi(\cdot)$.  
Recalling that \eqref{e:lin_decomp1}  implies the linear solution operator can be decomposed as
\begin{equation}\label{e:lin_decomp}
    e^{\mathcal{L}[\phi]t}f(x) = \phi'(x)\underbrace{\left(\frac{1}{N} \LA\widetilde{\Phi}_0,f\RA_{L^2_N} + s_{p,N}(t)f(x)\right)}_{\text{phase modulation}} 
    + \underbrace{\widetilde{S}(t)f(x)}_{\text{faster decaying residual}}
\end{equation}
it follows that we can remove the principle (i.e. slowest decaying) part of the nonlinear perturbation by implicitly defining
\begin{equation}\label{e:mod_slave1}
    \left\{\begin{aligned}
    &\gamma(t) \sim \LA\widetilde{\Phi}_0,v(0)\RA_{L^2_N} + \int_0^t \LA\widetilde{\Phi}_0,\mathcal{N}(s)\RA_{L^2_N} ds\\
    &\psi(x,t) \sim s_{p,N}(t)v(0) + \int_0^t s_{p,N}(t-s)\mathcal{N}(s)ds,
    \end{aligned}\right.
\end{equation}
where here $\sim$ indicates equality for $t\geq 1$.  This choice then yields the implicit description 
\begin{equation}\label{e:v1}
v(x,t) \sim \psi_x(x,t)v(x,t) + \widetilde{S}(t)v(0) + \int_0^t \widetilde{S}(t-s)\mathcal{N}(s) ds
\end{equation}
involving only the faster decaying residual component of the linear solution operator.  

Note the above choices clearly cannot extend all the way to $t=0$  due to an incompatibility of these choices with the initial data on $(v,\gamma,\psi)$.  
Here, we choose to keep the above choices for all $t\geq 1$ while interpolating
between the initial data and the right hand sides of \eqref{e:mod_slave1}-\eqref{e:v1} on the initial layer $0\leq t\leq 1$.
Specifically, we let $\chi(t)$ be a smooth cutoff function that is zero for $t\leq 1/2$ and one for $t\geq 1$, and define the modulation
functions $\gamma$ and $\psi$ implicitly for all $t\geq 0$ as
\begin{equation}\label{e:mod_slave2}
    \left\{\begin{aligned}
    &\gamma(t) = \chi(t)\left[ \LA\widetilde{\Phi}_0,v(0)\RA_{L^2_N} + \int_0^t \LA\widetilde{\Phi}_0,\mathcal{N}(s)\RA_{L^2_N} ds\right]\\
    &\psi(x,t) =\chi(t)\left[ s_{p,N}(t)v(0) + \int_0^t s_{p,N}(t-s)\mathcal{N}(s)ds\right],
    \end{aligned}\right.
\end{equation}
leaving the system
\begin{equation}\label{e:v2}
\begin{aligned}
v(x,t) &= \left(1-\chi(t)\right)\left[e^{\mathcal{L}[\phi]t}v(x,0)+\int_0^t e^{\mathcal{L}[\phi](t-s)}\mathcal{N}(s)ds\right]\\
&\quad+\chi(t)\left(\psi_x(x,t)v(x,t)+ \widetilde{S}(t)v(0) + \int_0^t \widetilde{S}(t-s)\mathcal{N}(s) ds\right).
\end{aligned}
\end{equation}
We note that from the differential equation \eqref{e:pert2}, along with the system of integral equations \eqref{e:mod_slave2}-\eqref{e:v2}, we readily obtain
short-time existence and continuity with respect to $t$ of a solution $(v,\psi_t,\psi_x)\in H^K_N$ and $\gamma\in W^{1,\infty}(0,\infty)$ 
by a standard contraction mapping argument, treating
\eqref{e:pert2} as a forced heat equation: see, for example, \cite{Hen}.  Associated with this solution, we now aim to obtain $L^2$ estimates on $(v,\gamma_t,\psi_x,\psi_t)$ and some of their
derivatives.

Noting that the nonlinear residual $\mathcal{N}$ in \eqref{e:pert2} involves
only derivatives of the modulation functions $\gamma$ and $\psi$, we may then expect to extract a closed system in $(v,\gamma_t,\psi_x,\psi_t)$, and some of their derivatives,
and then recover $\gamma$ and $\psi$ through the slaved system \eqref{e:mod_slave2}.  In particular, observe that using \eqref{e:v2} we see that control
of $v$ in, say, $L^2_N$ requires (in part) control $v$ in $H^2_N$.  This loss of derivatives is compensated by the following result, established by energy
estimates in \cite{JNRZ_13_1,JZ_11_1}, which uses the dissipative nature of the governing evolution equation to control higher derivatives of $v$ by lower ones, enabling us to close
our nonlinear iteration.

\begin{proposition}[Nonlinear Damping]\label{P:nonlin_damp}
Suppose the nonlinear perturbation $v$ defined in \eqref{e:nlin_residual} satisfies $v(\cdot,0)\in H^K_N$, and suppose that for some $T>0$ the $H^K_N$ norm of $v$ and $\psi_t$,
the $H^{K+1}_N$ norm of $\psi_x$, and the $L^\infty$ norms of $\gamma$ and $\gamma_t$ 
remain bounded by a sufficiently small constant for all $0\leq t\leq T$.  Then there exist positive constants $\theta,C>0$,
both independent of $N$ and $T$, such that
\[
    \|v(t)\|_{H^K_N}^2 \lesssim e^{-\theta t}\|v(0)\|_{H^K_N}^2 + \int_0^t e^{-\theta(t-s)}\left(\left\|v(s)\right\|_{L^2_N}^2 + \left\|\psi_x(s)\right\|_{H^{K+1}_N}^2 + \left\|\psi_t(s)\right\|_{ H^{K}_N}^2 + \left|\gamma_t(s)\right|^2\right)ds
\]
for all $0\leq t\leq T$.
\end{proposition}

\begin{proof}
The proof strategy is by now standard, and can be found, or example, in \cite{JNRZ_13_1,JZ_11_1}.  For completeness, here we simply outline the main details.  First, one
rewrites \eqref{e:pert2} as the forced heat equation
\begin{align*}
(1-\psi_x)\left(kv_t-k^2v_{xx}\right)&=-k\left(\psi_t+\frac{1}{N}\gamma_t\right)\phi'+k^2\left(\frac{\psi_x}{1-\psi_x}~\phi'\right)_x
-\psi_xf(\phi+v)+f(\phi+v)-f(\phi)\\
&\quad+kv_x\left(c-\psi_t-\frac{\gamma_t}{N}\right)
+k^2\left[\left(\frac{1}{1-\psi_x}+1\right)\psi_xv_x\right]_x.
\end{align*}
Multiplying by $\sum_{j=0}^K(-1)^j\frac{\partial_x^{2j}v}{1-\psi_x}$, integrating over $[0,N]$, using integration by parts and rearranging yields a bound of the form\footnote{Below,
the symbol $A\lesssim B$ implies there exists a constant $C>0$, independent of $N$, such that $A\leq CB$.}
\begin{align*}
\partial_t\|v\|_{H^K_N}^2+2k\|v\|_{H^{K+1}_N}^2&\lesssim \eps\|v\|_{H^{K+1}_N}^2+\|v\|_{L^2_N}^2+\frac{1}{\eps}\left\|\frac{\psi_t}{1-\psi_x}~\phi'\right\|_{H^{K-1}_N}^2\\
&+\frac{|\gamma_t|^2}{N^2\eps}\left\|\frac{1}{1-\psi_x}~\phi'\right\|_{H^{K-1}_N}^2+\frac{1}{\eps}\left\|\frac{1}{1-\psi_x}\partial_x\left(\frac{\psi_x}{1-\psi_x}~\phi'\right)\right\|_{H^{K-1}_N}^2\\
&+\frac{1}{\eps}\left\|\frac{\psi_x}{1-\psi_x}f(\phi+v)\right\|_{H^{K-1}_N}^2+\frac{1}{\eps}\left\|\frac{1}{1-\psi_x}\left(f(\phi+v)-f(\phi)\right)\right\|_{H^{K-1}_N}^2\\
&+\frac{1}{\eps}\left\|\frac{v_x}{1-\psi_x}\right\|_{H^{K-1}_N}^2+\frac{1}{\eps}\left\|\frac{\psi_tv_x}{1-\psi_x}\right\|_{H^{K-1}_N}^2
	+\frac{|\gamma_t|^2}{N^2\eps}\left\|\frac{v_x}{1-\psi_x}\right\|_{H^{K-1}_N}^2\\
&+\frac{1}{\eps}\left\|\frac{1}{1-\psi_x}\partial_x\left[\left(\frac{1}{1-\psi_x}+1\right)\psi_xv_x\right]\right\|_{H^{K-1}_N}^2,
\end{align*}
where here $\eps>0$ is an arbitrary constant\footnote{Introduced by the application of the Cauchy inequality with $\eps$ throughout.} independent of $N$.
Using the Sobolev interpolation
\[
\|g\|_{H^K_N}^2\leq \widetilde{C}^{-1}\|\partial_x^{K+1}g\|_{L^2_N}^2+\widetilde{C}\|g\|_{L^2_N}^2,
\]
valid for some constant $\widetilde{C}>0$ independent of $N$, now gives
\[
\frac{d}{dt}\|v\|_{H^K_N}^2(t)\leq -\theta\|v(t)\|_{H^{K}_N}^2+C\left(\|v(t)\|_{L^2_N}^2+\|\psi_x\|_{H^{K+1}_N}^2+\|\psi_t\|_{H^K_N}^2+|\gamma_t(t)|^2\right).
\]
The proof is now complete by an application of Gronwall's inequality.
\end{proof}

\subsection{Nonlinear Iteration}\label{S:nlin_iteration}

To complete the proof of Theorem \ref{T:main}, associated to the solution $(v,\gamma_t,\gamma_t,\gamma_x)$ of of \eqref{e:mod_slave2}-\eqref{e:v2}
we define, so long as it is finite, the function
\[
    \zeta(t) := \sup_{0\leq s\leq t}\left( \left\|v(s)\right\|_{H^K_N}^2 + \left\|\psi_x(s)\right\|_{H^{K+1}_N}^2 + \left\|\psi_t(s)\right\|_{ H^{K}_N}^2 + \left|\gamma_t(s)\right|\right)^{1/2}(1+s)^{3/4}.
\]
Combining the linear estimates in Proposition \ref{P:lin_est} with the damping estimate in Proposition \ref{P:nonlin_damp}, we now establish a key inequality for $\zeta$ which will
yield global existence and stability of our solutions.

\begin{proposition}\label{P:iteration}
Under the assumptions of Theorem \ref{T:main}, there exist positive constants $C,\eps>0$, both independent of $N$, such that if $v(\cdot,0)$ is such that
\[
E_0:=\|v(\cdot,0)\|_{L^1_N\cap H^K_N}\leq \eps\quad{\rm and}\quad\zeta(T)\leq \eps 
\]
for some $T>0$, then we have
\[
    \zeta(t) \leq C\left(E_0 + \zeta^2(t)\right) 
\]
valid for all $0\leq t\leq T$.
\end{proposition}

\begin{proof}
Recalling Lemma \ref{L:nlin_pert} we readily see that there exists a constant $C>0$, independent of $N$, such that
\[
    \|\mathcal{Q}(t)\|_{L^1_N\cap H^1_N} \leq C \left(1+\|\psi_x(t)\|_{H^1_N}\right)\|v(t)\|_{H^1_N}^2 
\]
and
\[
    \|\mathcal{R}(t)\|_{L^1_N\cap H^1_N}\leq C \left(\|(v, v_x, \psi_x, \psi_{xx}, \psi_t)(t)\|_{H^1_N}^2 + |\gamma_t(t)|^2\right) 
\]
so that, using the linear estimates in Proposition \ref{P:lin_est}, we have for so long as $\zeta(t)$ remains small that
\[
  \|\mathcal{Q}(t)\|_{L^1_N\cap H^1_N},~~ \|\mathcal{R}(t)\|_{L^1_N\cap H^1_N}\leq C \zeta^2(t)(1+t)^{-3/2}
\]
for some constant $C>0$ which is independent of $N$.  Since $k\mathcal{N} = \mathcal{Q} + k\mathcal{R}_x$, it follows there exists a constant $C>0$ independent of $N$ such that
\begin{equation}\label{e:Nbd}
    \|\mathcal{N}(t)\|_{L^1_N\cap H^1_N} \leq C \|(v, v_x, v_{xx}, \psi_x, \psi_{xx}, \psi_{xxx}, \psi_t, \psi_{tx})(t)\|_{H^1_N}^2 + |\gamma_t(t)|^2 \leq C \zeta^2(t)(1+t)^{-3/2}.
\end{equation}
for so long as $\zeta(t)$ remains small.  Applying the bounds in Proposition \ref{P:lin_est} to the implicit equation \eqref{e:v2}, it immediately follows 
that
\begin{align*}
\left\|v(t)\right\|_{L^2_N}&\leq \left\|v(\cdot,t)\psi_x(t)\right\|_{L^2_N}+CE_0(1+t)^{-3/4}
	+C\int_0^t(1+t-s)^{-3/4}\left\|\mathcal{N}(s)\right\|_{L^1_N\cap L^2_N}ds\\
&\leq \zeta(t)^2(1+t)^{-3/2}+CE_0(1+t)^{-3/4}+C\zeta(t)^2\int_0^t(1+t-s)^{-3/4}(1+s)^{-3/2}ds	\\
&\leq C\left(E_0+\zeta(t)^2\right)(1+t)^{-3/4}
\end{align*}
for some constant $C>0$ independent of $N$.  In particular, observe the loss of derivatives in the above estimate: control of the $L^2_N$ norm of 
$v(t)$ requires control of the $H^K_N$ norm of $v(t)$.  This loss of derivatives may be compensated by the nonlinear damping estimate in Proposition \ref{P:nonlin_damp},
assuming we can obtain appropriate estimates on the modulation functions and their derivatives.

To this end, we observe that  by using \eqref{e:mod_slave2} for $0\leq \ell\leq K+1$ we have that
\[
    \d_x^\ell\psi_x(x,t) = \chi(t)\left(\d_x^{\ell+1}s_{p,N}(t)v(0) + \int_0^t \d_x^{\ell+1}s_{p,N}(t-s)\mathcal{N}(s)ds\right),
\]
and for $0\leq \ell \leq K$
\begin{align*}
    \d_x^\ell\psi_t(x,t) &= \chi(t)\left(\d_x^{\ell}\d_t [s_{p,N}](t)v(0) + \d_x^\ell s_{p,N}(0)\mathcal{N}(t) + \int_0^t \d_x^{\ell}\d_t[s_{p,N}](t-s)\mathcal{N}(s)ds\right)\\
    		&\quad +\chi'(t)\left(\d_x^\ell s_{p,N}(t)v(0)+\int_0^t\d_x^{\ell}s_{p,N}(t-s)\mathcal{N}(s)ds\right),
\end{align*}
and hence that
\[
\left\|\psi_x\right\|_{H^{K+1}_N}, \left\|\psi_t\right\|_{H^{K}_N}\leq C\left(E_0+\zeta(t)^2\right)(1+t)^{-3/4}.
\]
Similarly, using \eqref{e:mod_slave2}(i) we find\footnote{Note here we use an $L^\infty$-$L^1$ bound to control the inner product.  This is opposed to using
Cauchy-Schwartz, which would contribute the growing factor $\|\widetilde{\Phi}_0\|_{L^2_N}=\mathcal{O}(N)$.}
\[
|\gamma_t(t)| =\left|\LA \widetilde{\Phi}_0, \mathcal{N}(t)\RA_{L^2_N}\right|\leq  C\|\mathcal{N}(t)\|_{L^1_N}\leq  C\left(E_0+ \zeta^2(t)\right)(1+t)^{-3/2}.
\]
Using the damping result in Proposition \ref{P:nonlin_damp}, we conclude that
\begin{equation}\label{e:vbd}
\begin{aligned}
    \|v(t)\|_{H^K_N}^2 &\leq C E_0^2 e^{-\theta t} + C\left(E_0 + \zeta^2(t)\right)^2\int_0^t e^{-\theta(t-s)} (1+s)^{-3/2}ds\\
    &\leq C E_0^2 e^{-\theta t} + C\left(E_0 + \zeta^2(t)\right)^2 (1+t)^{-3/2}\\
    &\leq C\left(E_0 + \zeta^2(t)\right)^2 (1+t)^{-3/2}.
\end{aligned}
\end{equation}
Since $\zeta(t)$ is a non-decreasing function, it follows that for a given $t\in(0,T)$ we have
\[
\left( \left\|v(s)\right\|_{H^K_N}^2 + \left\|\psi_x(s)\right\|_{H^{K+1}_N}^2 + \left\|\psi_t(s)\right\|_{ H^{K}_N}^2 + \left|\gamma_t(s)\right|^2\right)^{1/2}(1+s)^{3/4}
\leq C\left(E_0 + \zeta^2(t)\right)^2 
\]
valid for all $s\in(0,t)$.  Taking the supremum over $s\in(0,t)$ completes the proof.
\end{proof}

The proof of Theorem \ref{T:main} now follows by continuous induction.  Indeed, 
$\zeta(t)$ is continuous so long as it remains small, Proposition \ref{P:iteration} implies that if $E_0<\frac{1}{4C}$ then $\zeta(t)\leq 2CE_0$ for all $t\geq 0$.  Noting 
that $C>0$ is independent of $N$, this establishes the stability estimates \eqref{e:result1} from Theorem \ref{T:main} by taking
\[
\widetilde{\psi}(x,t):=\frac{1}{N}\gamma(t)+\psi(x,t).
\]
Further, the stability estimate \eqref{e:result3} in Corollary \ref{C:main} follows by \eqref{e:vbd} and the triangle inequality since
\begin{align*}
\left\|u\left(\cdot-\frac{1}{N}\gamma(t),t\right)-\phi\right\|_{L^2_N}&\leq \|u_x\|_{L^\infty}\|\psi(x,t)\|_{L^2_N}+CE_0(1+t)^{-3/4}\\
&\leq CE_0\left(1+t\right)^{-1/4},
\end{align*}
as claimed.  Further, note that since for $0<t<s$ we have
\[
|\gamma(t)-\gamma(s)|\leq\int_t^s|\gamma_t(z)|dz\leq CE_0(1+t)^{-1/2}
\]
it follows that $\gamma(t)$ converges to some\footnote{Note since the modulation function $\gamma$ depends
on $N$, so does the limiting phase shift $\gamma_\infty$.} $\gamma_\infty\in\RM$ as $t\to\infty$ with rate 
\[
|\gamma(t)-\gamma_\infty|\leq\int_t^\infty|\gamma_t(z)|dz\leq CE_0(1+t)^{-1/2},
\]
which, by the triangle inequality, establishes \eqref{e:result2}, thus completing the proof of Theorem \ref{T:main}, as well as 
completes the proof of Corollary \ref{C:main}.
In fact, notice that from \eqref{e:Nbd} we have
\[
\left|\int_0^t\left<\widetilde{\Phi}_0,\mathcal{N}(s)\right>_{L^2_N}ds\right| \leq  C\left\|\widetilde{\Phi}_0\right\|_{L^\infty(\mathbb{R})}\zeta^2(t)\int_0^t (1+s)^{-3/2}ds,
\]
which, since the above work shows that $\zeta(t)\leq C\eps$ for some constant $C>0$, implies from \eqref{e:mod_slave2} that
\[
\gamma_\infty = \left<\widetilde{\Phi}_0,v(0)\right>_{L^2_N} + \mathcal{O}(\eps^2).
\]
That is, the asymptotic phase shift in Theorem \ref{T:main} is $\mathcal{O}(\eps^2)$ close to that  suggested by the linear theory in Section \ref{S:lin_stab}.

Finally, we combine Corollary \ref{C:main} with Proposition \ref{P:sub_stab} to establish Corollary \ref{C:min_thm}.  To this end, 
let $\eps>0$ and $C>0$ be as in Corollary \ref{C:main}.  Fix $N\in\NM$ and $\delta\in(0,\delta_N)$, with $\delta_N$ as in \eqref{spec_gap}, and let $\eps_\delta>0$ be as in 
Proposition \ref{P:sub_stab}.   If $u_0\in L^1_{\rm per}(0,N)\cap H^K_{\rm per}(0,N)$ with $E_0<\eps$, then Corollary \ref{C:main} implies that
\[
\left\|u\left(\cdot,t\right)-\phi\left(\cdot+\frac{1}{N}\gamma_\infty\right)\right\|_{H^1_{\rm per}(0,N)}\leq CE_0 (1+t)^{-1/4}
\]
for all $t>0$.  In particular, there exists a time $T_\delta>0$ such that
\[
\left\|u\left(\cdot,t\right)-\phi\left(\cdot+\frac{1}{N}\gamma_\infty\right)\right\|_{H^1_{\rm per}(0,N)}<\eps_\delta
\]
for all $t\geq T_\delta$.  By the translational invariance of \eqref{e:RDE_trav} it is clear that $\phi(\cdot+\frac{1}{N}\gamma_\infty)$ is a diffusively
spectrally stable $1$-periodic solution of \eqref{e:rd}, and hence Proposition \ref{P:sub_stab} implies\footnote{Here, we are applying Proposition \ref{P:sub_stab}
with initial data $u(\cdot,T_\delta)$.} there exists a constant
$C_\delta>0$ such that
\[
\left\|u\left(\cdot,t\right)-\phi\left(\cdot+\frac{1}{N}\gamma_\infty\right)\right\|_{H^1_{\rm per}(0,N)}\leq C_\delta\eps_\delta e^{-\delta t}
\]
for all $t>T_\delta$.  Taking $M_\delta=\frac{C_\delta\eps_\delta}{E_0}$ completes the proof.

\appendix
\section{Bounds on Discrete Sums}

In order to establish the uniform linear bounds in Proposition \ref{P:lin_est}, we need to establish uniform-in-$N$ bounds on finite sums of the form
\[
    \frac{1}{N}\sum_{\xi\in\Omega_N\setminus\{0\}}\xi^{2r} e^{-2d\xi^2 t}
\]
where $N\in\NM$.  Following the ideas in \cite{HJP_1}, we note that the above finite sum is, up to a simple rescaling, a Riemann
sum approximation for the integral 
\[
    \int_{-\pi}^{\pi} \xi^{2r}e^{-2d\xi^2 t}d\xi
\]
which, through an elementary scaling argument, exhibits $(1+t)^{-r-1/2}$ decay for large time.  Using this as  motivation, we now establish
the following key estimate.

\begin{lemma}\label{L:sum_poly_bd}
Let $d>0$ and $r\in\N\cup\{0\}$ be given.  Then there exists a constant $C>0$, independent of $N$, such that for every $N\in\NM$ we have
\[
\frac{1}{N}\sum_{\xi\in\Omega_N\setminus\{0\}}\xi^{2r} e^{-2d\xi^2 t}\leq C(1+t)^{-r-1/2},
\]
valid for all $t\geq 0$.
\end{lemma}

\begin{proof}
First, consider the case when $r=0$ and note that, for each $t>0$, the function $\xi\mapsto e^{-2d\xi^2t}$ is even and monotonically decreasing for $\xi>0$.  
Together with the equality $\xi_j - \xi_{j-1} = 2\pi/N$, monotonicity allows us to treat the sum over $\xi\in\Omega_N$, $\xi>0$ as a right-endpoint Riemann sum (i.e. an under-approximation).  Parity then tells us the sum over $\xi\in\Omega_N$, $\xi<0$ is also an under-approximation, yielding
\[
    \frac{1}{N}\sum_{\xi\in\Omega_N\setminus\{0\}}e^{-2d\xi^2 t} \leq \frac{1}{2\pi} \int_{-\pi}^{\pi} e^{-2d\xi^2 t}d\xi \lesssim (1+t)^{-1/2}.
\]  

For $r\geq 1$, the analysis is complicated by the fact that the function
\[
    f(\xi,t) := \xi^{2r}e^{-2d\xi^2 t},
\]
defined for $\xi\in\RM$ and $t>0$, is not monotonically decreasing for $\xi>0$.  However, we may use similar analysis via the following procedure.

First, observe that, for fixed $t>0$, $f(\cdot,t)$ has a global minimum at 0 and global maxima at
\[
    \pm R := \pm\left(\frac{r}{2d}\right)^{1/2}t^{-1/2},~~{\rm with}~~ f(\pm R,t) = \left(\frac{r}{2de}\right)^{r}t^{-r}.
\]
If $0 < t \leq r/(2d\pi^2)$, then $R\geq \pi$ so that $\pm R \notin (-\pi,\pi)$.  We can then easily estimate the sum
\begin{equation}\label{e:Rsum_shorttime}
    \frac{1}{N}\sum_{\xi\in\Omega_N\setminus\{0\}} \xi^{2r}e^{-2d\xi^2 t} \leq \frac{1}{N} \sum_{\xi\in\Omega_N\setminus\{0\}} \pi^{2r} \leq \pi^{2r}.
\end{equation}
For $t > r/(2d\pi^2)$, we define the auxiliary function
\[
    G(\xi,t) :=\begin{cases}
    f(R,t), & |\xi|\leq R\\
    f(\xi,t), & |\xi|> R
    \end{cases}.
\]
Notice that $G(\cdot,t)$ is even and monotonically decreasing for $\xi>0$.  Furthermore, notice that 
\[
    \int_{-\pi}^{\pi} G(\xi,t)\, d\xi \leq 2e^{1/2}\left(\frac{r}{2de}\right)^{r+1/2}t^{-r-1/2} + \int_{-\pi}^{\pi} f(\xi,t)\, d\xi \lesssim (1+t)^{-r-1/2},
\]
where the last inequality follows from \eqref{e:Rsum_shorttime}.  Consequently, we may modify the monotonicity trick from the $r=0$ case to obtain
\[
    \frac{1}{N}\sum_{\xi\in\Omega_N\setminus\{0\}} \xi^{2r}e^{-2d\xi^2 t} \leq \frac{1}{2\pi} \int_{-\pi}^{\pi} G(\xi,t)\, d\xi \lesssim (1+t)^{-r-1/2}.
\]
\end{proof}

In \cite{HJP_1}, the authors further established that, in the cases $r=0$ and $r=1$, the decay rate in Lemma \ref{L:sum_poly_bd} is indeed sharp, providing also a uniform lower bound for
the corresponding finite sums.  A similar analysis applied to the present situation establishes the sharpness of these bounds for all $r\geq 0$.  While not necessary in the present
analysis, it provides yet a deeper connection between the current uniform analysis of subharmonic perturbations and the ``limiting'' localized theory.

\bibliographystyle{abbrv}
\bibliography{RD}

\begin{thebibliography}{10}

\bibitem{DS_18}
B.~de~Rijk and B.~Sandstede.
\newblock Diffusive stability against nonlocalized perturbations of planar wave
  trains in reaction-diffusion systems.
\newblock {\em Journal of Differential Equations}, 265:5315--5351, 2018.

\bibitem{DSSS}
A.~Doelman, B.~Sandstede, A.~Scheel, and G.~Schneider.
\newblock The dynamics of modulated wave trains.
\newblock {\em Mem. Amer. Math. Soc.}, 199(934):viii+105, 2009.

\bibitem{G93}
R.~Gardner.
\newblock On the structure of the spectra of periodic traveling waves.
\newblock {\em J. Math. Pures Appl.}, 72:415--439, 1993.

\bibitem{HJP_1}
M.~Haragus, M.~A. Johnson, and W.~R. Perkins.
\newblock Linear modulational and subharmonic dynamics of spectrally stable
  {L}ugiato-{L}efever periodic waves.
\newblock {\em Journal of Differential Equations}, in press, 2021.

\bibitem{Hen}
D.~Henry.
\newblock {\em Geometric theory of semilinear parabolic equations}, volume 840
  of {\em Lecture Notes in Mathematics}.
\newblock Springer-Verlag, Berlin-New York, 1981.

\bibitem{JNRZ_13_1}
M.~A. Johnson, P.~Noble, L.~M. Rodrigues, and K.~Zumbrun.
\newblock Nonlocalized modulation of periodic reaction diffusion waves:
  nonlinear stability.
\newblock {\em Arch. Ration. Mech. Anal.}, 207(2):693--715, 2013.

\bibitem{JNRZ_13_2}
M.~A. Johnson, P.~Noble, L.~M. Rodrigues, and K.~Zumbrun.
\newblock Nonlocalized modulation of periodic reaction diffusion waves: the
  {W}hitham equation.
\newblock {\em Arch. Ration. Mech. Anal.}, 207(2):669--692, 2013.

\bibitem{JNRZ_Invent}
M.~A. Johnson, P.~Noble, L.~M. Rodrigues, and K.~Zumbrun.
\newblock Behavior of periodic solutions of viscous conservation laws under
  localized and nonlocalized perturbations.
\newblock {\em Inventiones Mathematicae}, 197(1):115--213, 2014.

\bibitem{JZ_11_1}
M.~A. Johnson and K.~Zumbrun.
\newblock Nonlinear stability of spatially-periodic traveling-wave solutions of
  systems of reaction-diffusion equations.
\newblock {\em Ann. Inst. H. Poincar\'{e} Anal. Non Lin\'{e}aire},
  28(4):471--483, 2011.

\bibitem{KP_book}
T.~Kapitula and K.~Promislow.
\newblock {\em Spectral and dynamical stability of nonlinear waves}, volume 185
  of {\em Applied Mathematical Sciences}.
\newblock Springer, New York, 2013.
\newblock With a foreword by Christopher K. R. T. Jones.

\bibitem{RS4}
M.~Reed and B.~Simon.
\newblock {\em Methods of modern mathematical physics. {IV}. {A}nalysis of
  operators}.
\newblock Academic Press [Harcourt Brace Jovanovich, Publishers], New
  York-London, 1978.

\bibitem{SSSU}
B.~Sandstede, A.~Scheel, G.~Schneider, and H.~Uecker.
\newblock Diffusive mixing of periodic wave trains in reaction-diffusion
  systems.
\newblock {\em J. Differential Equations}, 252(5):3541--3574, 2012.

\bibitem{SW15}
A.~Scheel and Q.~Wu.
\newblock Diffusive stability of {T}uring patterns via normal forms.
\newblock {\em J. Dynam. Differential Equations}, 27(3-4):1027--1076, 2015.

\bibitem{S96}
G.~Schneider.
\newblock Diffusive stability of spatial periodic solutions of the
  {S}wift-{H}ohenberg equation.
\newblock {\em Comm. Math. Phys.}, 178(3):679--702, 1996.

\bibitem{S98_1}
G.~Schneider.
\newblock Nonlinear diffusive stability of spatially periodic
  solutions---abstract theorem and higher space dimensions.
\newblock In {\em Proceedings of the {I}nternational {C}onference on
  {A}symptotics in {N}onlinear {D}iffusive {S}ystems ({S}endai, 1997)},
  volume~8 of {\em Tohoku Math. Publ.}, pages 159--167. Tohoku Univ., Sendai,
  1998.

\end{thebibliography}
\end{document}